\documentclass[11pt]{amsart}
\usepackage{amsmath}
\usepackage{amssymb}
\usepackage{amsthm}
\usepackage{verbatim}

\newtheorem{thm}{Theorem}[section]

\newtheorem{lem}[thm]{Lemma}

\theoremstyle{definition}

\theoremstyle{remark}
\newtheorem{rem}[thm]{Remark}

\def \S{\mathrm{S}}

\def \I{\mathrm{I}}

\def \C{\mathbb{C}}

\def\R{\mathbb{R}}
\def\2L{\Lambda_{\tilde{\gamma}}}
\def\1L{\Lambda_{\gamma}}

\renewcommand{\le}{\leqslant}
\renewcommand{\ge}{\geqslant}

\def \FF{\mathcal{F}}
\begin{document}

\title[A nonlinear Plancherel identity]{On Plancherel's identity for a two--dimensional scattering transform}

\thanks{Supported by Academy of Finland project  SA-12719831, by MINECO grants MTM2011-28198, MTM2013-41780-P and SEV-2011-0087 (Spain) and ERC grant 307179}
\author{Kari Astala}
\address{Le Studium, Loire Valley Institute for Advanced Studies, Orl\'eans \& Tours, France; Mapmo, University of Orl\'eans, rue de Chartres, 45100 Orl\'eans, France; Department of mathematics and statistics, Po.Box 68, 00014 University of Helsinki, Finland} \email{kari.astala@helsinki.fi}
\author{Daniel Faraco}
\address{Departamento de Matem\'aticas - Universidad Aut\'onoma de Madrid and Instituto de Ciencias Matem\'aticas CSIC-UAM-UC3M-UCM, 28049 Madrid, Spain} \email{daniel.faraco@uam.es}
\author{Keith M. Rogers}
\address{Instituto de Ciencias Matem\'aticas CSIC-UAM-UC3M-UCM, 28049 Madrid, Spain} \email{keith.rogers@icmat.es}

\maketitle

\begin{abstract} We consider the $\overline{\partial}$-Dirac system that Ablowitz and Fokas used to transform the defocussing Davey--Stewartson system to a linear evolution equation. The nonlinear Plancherel identity for the associated scattering transform was established by Beals and Coifman for Schwartz functions. Sung extended the validity of the identity to functions belonging to $L^1\cap L^\infty(\mathbb{R}^2)$ and Brown to $L^2$-functions with sufficiently small norm. More recently, Perry extended to the weighted Sobolev space $H^{1,1}(\mathbb{R}^2)$ and here we extend to  $H^{s,s}(\mathbb{R}^2)$ with $0<s<1$.
 \end{abstract}

\vspace{1em}

\section{Introduction}

Plancherel's identity for the Fourier transform $\, \widehat{\ }$ defined initially on Schwartz functions by
$$
\widehat{F}(\xi)=\frac{1}{(2\pi)^{n/2}}\int_{\R^n} F(x)\,e^{-i \xi\cdot x}dx,
$$
states that 
$$
\|\widehat{F}\|_2=\|F\|_2.
$$
Using the linearity of the transform, this also yields Lipschitz continuity
\begin{equation}\label{lc}
\|\widehat{F}-\widehat{G}\|_2\le \|F-G\|_2,
\end{equation}
and so the Fourier transform of Cauchy sequences are Cauchy sequences, which have limits in $L^2$. This allows us to define the transform for all $L^2$ functions and extend the identity to this larger class.

In this note we consider the two-dimensional nonlinear Fourier transform $\FF $ associated with the $\overline{\partial}$-Dirac system, first considered by Ablowitz and Fokas \cite{AF, FA}  (see the following section for the precise definition). Beals and Coifman \cite{BC0, BC1,BC} established Plancherel's identity
$$
\|\FF [F]\|_2=\|F\|_2
$$
for Schwarz functions $F$.
  Unlike in the linear case, the Lipschitz continuity is not an immediate consequence of the identity and so this is not enough to extend to~$L^2$. Sung treated $F\in L^1\cap L^\infty(\R^2)$ in \cite{S1}, then Brown~\cite{brown} proved Lipschitz continuity for $L^2$-functions with sufficiently small norm, allowing him to extend the definition of $\FF $ and the validity of the Plancherel identity to these functions. In~\cite{perry}, Perry  proved a local version of Lipschitz continuity for functions in the class $H^{1,1}$, where
$$
\|F\|_{H^{s,s}}=\max\{\|F\|_{L^2(\R^2)}, \|D^sF\|_{L^2(\R^2)},\||\cdot|^sF\|_{L^2(\R^2)}\},
$$
and $\widehat{D^sF}=|\cdot|^s \widehat{F}$; in Perry's case $D^1$ can of course be replaced by the gradient. He then applied his result to the Davey--Stewartson system, as did the previously mentioned authors, proving global well-posedness in $H^{1,1}$ for the defocussing system. Indeed, the main motivation for studying this scattering problem (there are many other scattering models of course, see for example \cite{CK}, \cite{MTT}) is that it solves the defocussing Davey--Stewartson system in the same way as the linear nonelliptic Schr\"odinger equation can be solved using the Fourier transform. This was first observed by Ablowitz and Fokas \cite{AF, FA}. That there can be blow--up in the focussing Davey--Stewartson system is due to Ozawa \cite{O}.

Here we refine part of Perry's argument, proving a substitute for  Lipschitz continuity for the functions in $H^{s,s}$ with $s\in(0,1)$, allowing us to extend the validity of the Plancherel identity to this space.

 In the following section we will recall the definition of the scattering transform and give a self-contained proof of the Plancherel identity for Schwartz functions. As in the case of the linear Fourier transform, there are slightly different conventions regarding the definition which consist of little more than changes in parameters -- we use the same normalisation as Sung~\cite{S1}. In the third section, we will prove the key technical lemma and recall how the arguments of Brown and Perry achieve the local Lipschitz continuity.

\section{The Plancherel identity for Schwartz functions}
In this section we take $q$ in  the Schwartz class and start by defining the scattering solutions $u$. We then define the   scattering transform $\mathcal{F}[q]$, and prove the Plancherel identity for Schwartz functions.

As $e^{i\overline{k}z/2}$ is a holomorphic function of $z\in \mathbb{C}$, the Cauchy--Riemann equations tell us that $\partial_{\overline{z}} e^{i\overline{k}z/2}=0$, where $\partial_{\overline{z}}=\frac{1}{2}(\frac{d}{dz_1}+i\frac{d}{dz_2})$. In order to find similar vector-valued solutions to the system
\begin{equation}\label{scat0}
\partial_{\overline{z}}\psi=Q\overline{\psi},\qquad Q=\left( \begin{array}{cc}
0 & q  \\
q & 0
 \end{array} \right),
\end{equation}
we require that $\lim_{|z|\to\infty}\psi(z,k)e^{-i\overline{k}z/2}\to (1,0)$. Writing $\psi=e^{i\overline{k}z/2}U$ and $e_k(z)=e^{i(\overline{k}z+k\overline{z})/2}=e^{ik\cdot z}$, this is equivalent to solving the system
\begin{align}\label{scat}
\partial_{\overline{z}}u_{1}&=e_{-k}q\overline{u_{2}}\\\nonumber
\partial_{\overline{z}}u_{2}&=e_{-k}q\overline{u_{1}},
\end{align}
where $\lim_{|z|\to\infty}u_1(z,k)\to 1$ and $\lim_{|z|\to\infty}u_2(z,k)\to 0$. 

We first recall why the solutions to this system are unique. By writing $v_1=u_1+u_2$ and $v_2=u_1-u_2$, we can add and subtract  to obtain the equivalent pair of equations
\begin{eqnarray}
\label{yksi} \partial_{\overline{z}}v_1&=e_{-k}q\overline{v_1}\\
\label{kaksi} \partial_{\overline{z}}v_2&=-e_{-k}q\overline{v_2}.
\end{eqnarray}
By a Liouville-type theorem (see for example \cite[Theorem 8.5.1]{AIM}),  bounded solutions to \eqref{yksi} or \eqref{kaksi}
have the form
\begin{equation}\label{kolme}
v(z) = C e^{\phi(z)}, \qquad \phi \in C_0(\C).
\end{equation}
In particular, by the linearity of the equations,  bounded solutions $v_1,v_2$ with the same given limit at infinity are unique. 
By adding and subtracting we see that $u_1,u_2$ are also unique.

For existence we appeal to Fredholm theory. By defining the operator 
\begin{equation*}\label{operator}
\S^k_q [F]=\partial_{\overline{z}}^{-1}\Big[e_{-k}
q\,\partial_{z}^{-1}\big[e_{k}\overline{q}F\big]\Big],
\end{equation*}
we see from \eqref{scat} that
$(\I-\S^{k}_{q})u_{1}=1$ and $(\I-\S^{k}_{q})u_{2}=\partial_{\overline{z}}^{-1}[e_{-k}q]$ and so it will suffice to invert $(\I-\S^{k}_{q})$ to obtain solutions. To see this, we require a number of well-known properties of the Cauchy transform. For example, by \cite[Theorems 4.3.11]{AIM}, we have the uniform bound
$$
\|\S^{k}_{q}[F]\|_\infty\le C_p\|q\|_p\|q\|_{p'}\|F\|_\infty\le C_s\|q\|_{H^{s,s}}^2\|F\|_\infty,\quad s>1-\tfrac{2}{p}>0.
$$
Equicontinuity can be deduced from \cite[Theorem 4.3.13]{AIM}, so by the Arzel\`a--Ascoli theorem $\S^{k}_{q}$ is a compact operator on $L^\infty(\R^2)$. In order to prove that $(\I-\S^{k}_{q})$ is injective we suppose that $(\I-\S^{k}_{q})[F]=F$ which is equivalent to
$$
F=\partial_{\overline{z}}^{-1}[e_{-k}q\overline{G}]\quad \text{where}\quad G=\partial_{\overline{z}}^{-1}[e_{-k}q\overline{F}].
$$
By adding and subtracting to obtain an equivalent system, we can argue as in \eqref{yksi}-\eqref{kolme} to see that $F$ is zero and so $(\I-\S^{k}_{q})$ is injective. 
Thus, by the Fredholm alternative, $\I-\S^k_{q}$ can be inverted so that $$u_{1}(z,k)=(\I-\S^k_{q})^{-1}[1](z),\qquad u_{2}=(\I-\S^k_{q})^{-1}\big[\partial_{\overline{z}}^{-1}[e_{-k}q]\big](z),$$ 
and we have obtained our scattering solutions.

We now make the simplifying assumption that $q$ has compact support, although we will soon see that this is unnecessary using arguments due to Sung \cite{S1}.  From \eqref{scat} we see that $u_1$ and $u_2$ are  holomorphic near infinity and so they have asymptotics given by their Laurent series;
\begin{align*}
u_1(z,k)&=1+a(k)z^{-1}+\sum_{j\ge2}a_j(k)z^{-j}\\
u_2(z,k)&=b(k)z^{-1}+\sum_{j\ge2}b_j(k)z^{-j},
\end{align*}
for $z$ outside a disk containing $\text{supp}(q)$.
The development combined with \eqref{scat} and the Cauchy formula allows us to define the scattering transform  by 
\begin{equation}\label{kuusi}
\mathcal{F}[q](k):=\tfrac{i}{2} b(k)=\frac{i}{2\pi}\int_{\R^2} e_{-k}(z)\, q(z)\, \overline{u_1(z,k)}\, dz.
\end{equation}
With this quantity we will be able to invert the process.
 
To see this, we consider 
$
(\partial_k\psi_{1}, \partial_{\overline{k}}\psi_{2})$, as functions of~$z$,
which also solve~\eqref{scat0}. Note that the pair can be written as $$e^{i\overline{k}z/2}(\partial_k u_1,\frac{iz}{2}u_2+\partial_{\overline{k}}u_2),$$ and  by differentiating the asymptotic series above, we see that \begin{equation}\label{asym}\lim_{|z|\to\infty}\big(\partial_k u_1,\frac{iz}{2}u_2+\partial_{\overline{k}}u_2\big)=(0,\mathcal{F}[q]).\end{equation} Therefore, by uniqueness of solutions to \eqref{scat0}, we have that 
\begin{align}\label{theform}
\partial_{k}\psi_1&=\overline{\mathcal{F}[q]}\psi_2\\\nonumber
\partial_{\overline{k}}\psi_2&=\mathcal{F}[q]\psi_1.\end{align}

This system takes a form similar to \eqref{scat0} in terms of the $\partial_k$-derivative. Indeed, writing $\phi_1=\psi_1$, $\phi_2=
\overline{\psi_{2}}$ and taking the complex conjugate of the second equation, it is equivalent to
\begin{equation}\label{scat2}
\partial_{k}\phi=\overline{\mathcal{F}[Q]} \, \overline{\, \phi \,},\qquad \overline{\mathcal{F}[Q]}=\left( \begin{array}{cc}
0 & \overline{\mathcal{F}[q]}  \\
\overline{\mathcal{F}[q]} & 0
 \end{array} \right).
\end{equation}
Note that the first coordinate $\phi_1=\psi_1$ is the same for both equations \eqref{scat0} and \eqref{scat2}.

By formally repeating the argument we return to the original form; 
\begin{equation}\label{scat23}
\partial_{\overline{z}}\psi=\overline{\mathcal{F}\circ\overline{\mathcal{F}[Q]}}\, \overline{\psi},\qquad \overline{\mathcal{F}\circ\overline{\mathcal{F}[Q]}}=\left( \begin{array}{cc}
0 & \overline{\mathcal{F}\circ\overline{\mathcal{F}[q]}}  \\
\overline{\mathcal{F}\circ\overline{\mathcal{F}[q]}} & 0
 \end{array} \right).
\end{equation}
This yields the inversion formula since, arguing as in \eqref{yksi}-\eqref{kolme} the function $\psi_1 + \psi_2$ is  nonvanishing, and 
$$
q=\frac{\partial_{\overline{z}}(\psi_1 + \psi_{2})}{\overline{\psi_1} + \overline{\psi_{2}}}=\overline{\mathcal{F}\circ\overline{\mathcal{F}[q]}}.
$$

We remark that this formal repetition of the argument is not yet rigorous  as we have not analysed the $k$-asymptotics -- we cannot use Laurent series as $\mathcal{F}[q]$ does not have compact support as a function of $k$. 
However we now   avoid  the use of Laurent series, by differentiating 
 \begin{align*}
 u_{1}&=1+\partial_{\overline{z}}^{-1}[e_{-k}q\overline{u_{2}}],\\
  u_{2}\nonumber
&=\partial_{\overline{z}}^{-1}[e_{-k}q\overline{u_{1}}],
\end{align*}
 equivalent to \eqref{scat}, and using their first order asymptotics and uniqueness properties to prove \eqref{theform} as in \cite{S1}. Indeed, by writing $$e_k(z)u_{2}(z)= \partial_{\overline{z}}^{-1}[e_{-k}(\cdot-z)q\overline{u_{1}}]$$ and differentiating, we obtain
 \begin{align}\label{derivatives}
\partial_ku_{1}&=\partial_{\overline{z}}^{-1}\big[e_{-k}q\overline{e_{-k}\partial_{\overline{k}}[e_ku_{2}]}\big]\\
e_{-k}\partial_{\overline{k}}[e_ku_{2}]&=\mathcal{F}[q]+\partial_{\overline{z}}^{-1}[e_{-k}q\overline{\partial_ku_{1}}],
\end{align}
where, in the noncompactly supported case,  $\mathcal{F}[q]$ is defined directly via the integral in~\eqref{kuusi}. Differentiating under the integral can be justified using  \cite[Lemma A.5]{S1}. By differentiating more directly, we also have that $e_{-k}\partial_{\overline{k}}[e_ku_{2}]=\frac{iz}{2}u_2+\partial_{\overline{k}}u_2$ which yields \eqref{asym}. Then by uniqueness of solutions to \eqref{scat0} as before, we obtain \eqref{theform} and~\eqref{scat2}.

In order to repeat the argument in \eqref{scat2} in the $k$ variable, with $q$ replaced by $\mathcal{F}[q]$, we need to show that $\lim_{|k|\to\infty}\phi(z,k)e^{-i\overline{k}z/2}\to (1,0)$. This will follow from estimates in the next section  
  (see the forthcoming Remark~\ref{rem}). That $\mathcal{F}[q] \in L^1 \cap L^{\infty}$ 
 (needed to use \cite[Theorem 8.5.1]{AIM} and to be sure that the integral in the definition of $\mathcal{F}\circ\overline{\mathcal{F}[Q]}$ is well-defined), follows by repeated integration by parts in \eqref{kuusi}, where we write $e_{-k} = 2i \partial_{z} (e_{-k}/\overline{k})$ and use \eqref{scat} and \eqref{kolme} at each iteration, generating as much decay in $|k|$ as is needed.

Plancherel's identity then follows as usual using Fubini's theorem and the inversion formula;
\begin{align*}
\|\mathcal{F}[q]\|_2^2&=\int_{\R^2} \mathcal{F}[q](k)\, \overline{\mathcal{F}[q]}(k)\, dk\\
&=\int_{\R^2} \frac{i}{2\pi} \int_{\R^2} e_{-k}(z)\,q(z)\,  \overline{u_{1}(z,k)}\, dz\, \overline{\mathcal{F}[q]}(k)\, dk\\
&=\int_{\R^2} q(z)\,\mathcal{F}\circ\overline{\mathcal{F}[q]}(z)\, dz\\
&=\int_{\R^2} q(z)\,\overline{q(z)}\, dz=\|q\|_2^2.
\end{align*}
Here we use that the factor $\overline{u_{1}(z,k)}$ is the same when defined with respect to $q$ or $\overline{\mathcal{F}[q]}$.

In order to extend the identity to potentials that are not in the Schwartz class, it remains to prove a substitute for the Lipschitz continuity~\eqref{lc}.

\section{Local Lipschitz continuity}

From the previous section we know that the scattering transform of Schwartz potentials could be \lq measured' in order to recover the potential, via the inversion formula. Then a substitute for the inequality
$$
\|F-G\|_2=\|(\widehat{F})^\vee-(\widehat{G})^\vee\|_2\le \|\widehat{F}-\widehat{G}\|_2
$$
would also tell us that this recovery process  is in some sense stable; two similar scattering transforms must have been produced by two similar potentials.

We will require the following lemma.

\begin{lem}\label{isitgood} Let $s\in[0,1]$ and $p\in(2,\infty)$. Then
$$
\|\S^k_q [F]\|_{L^p}\le C(1+|k|)^{-s}\|q\|^2_{H^{s,s}}\|F\|_{L^{p}}
$$
and
$$
\|\S^k_q [F]\|_{L^p}\le C(1+|k|)^{-s}\|q\|^2_{H^{s,s}}\|F\|_{L^{p+s}}.
$$
Moreover, if $p\in (2/s,\infty)$, then
$$
\|\S^k_q [F]\|_{L^p}\le C(1+|k|)^{-s}\|q\|^2_{H^{s,s}}\|F\|_\infty.
$$
\end{lem}

\begin{proof} It will suffice to prove that if $s>2(\frac{1}{p}-\frac{1}{r})$, then
$$
\|\S^k_q [F]\|_{L^p}\le C(1+|k|)^{-s}\|q\|^2_{H^{s,s}}\|F\|_{L^{r}}.
$$
To deal with the first part of the operator,  we will require the estimate
\begin{equation}\label{isitgood0}
\big\|\partial_{\overline{z}}^{-1}\big[e_{-k}G\big]\big\|_{p}\le C(1+|k|)^{-s}\Big(\|G\|_{\frac{2p}{p+2}}+\|D^sG\|_{\frac{2p}{p+2}}\Big),
\end{equation}
which follows easily by complex interpolation between 
$$
\big\|\partial_{\overline{z}}^{-1}\big[e_{-k}G\big]\big\|_{p}\le C\|G\|_\frac{2p}{p+2},
$$
which is a consequence of the Hardy--Littlewood--Sobolev inequality (see for example \cite[Theorem 1.2.16 or Theorem 6.1.3]{G}) as the kernel of $\partial_{\overline{z}}^{-1}$ is $\frac{1}{\pi z}$, 
and
\begin{equation}\label{per}
\big\|\partial_{\overline{z}}^{-1}\big[e_{-k}G\big]\big\|_{p}\le C|k|^{-1}\|D G\|_\frac{2p}{p+2}.
\end{equation}
For this second inequality, we note that the Fourier multiplier associated to $G\mapsto e_{k}\partial_{\overline{z}}^{-1}\big[e_{-k}G\big]$  can be written as
$$
\frac{2}{i(\xi_1-k_1)-(\xi_2-k_2)}=\frac{2}{ik_1-k_2}\Big(\frac{i\xi_1-\xi_2}{i(\xi_1-k_1)-(\xi_2-k_2)}-1\Big),
$$
so that
$$
\big|\partial_{\overline{z}}^{-1}\big[e_{-k}G\big]\big|\le \frac{2}{|k|}\Big(\big|\partial_{\overline{z}}^{-1}\big[e_{-k}\partial_{\overline{z}} G\big]\big|+|G|\Big),
$$
and then we apply the Hardy--Littlewood--Sobolev inequality again. Note that $$\|\partial_{\overline{z}} G\|_\frac{2p}{p+2}\le C\|DG\|_\frac{2p}{p+2}. $$
See formula (2.8) of \cite{perry} for a proof of \eqref{per} by integration by parts.

Now after applying \eqref{isitgood0}, we are required to prove
$$
\Big\|D^s\Big[
q\,\partial_{z}^{-1}\big[e_{k}\overline{q}F\big]\Big]\Big\|_{\frac{2p}{p+2}}\le C\|q\|^2_{H^{s,s}}\|F\|_{L^{r}}
$$
and 
$$
\big\|
q\,\partial_{z}^{-1}\big[e_{k}\overline{q}F\big]\big\|_{\frac{2p}{p+2}}\le C\|q\|^2_{H^{s,s}}\|F\|_{L^{r}}.
$$
By the fractional Leibnitz rule (see for example \cite{B} or the appendix of \cite{KPV}) for the first, and H\"older's inequality for the second, these estimates would follow from
$$
\big\|\partial_{z}^{-1}\big[e_{k}\overline{q}F\big]\big\|_p\le C\|q\|_{H^{s,s}}\|F\|_{L^r}
$$
and 
$$
\big\|D^{s}\partial_{z}^{-1}\big[e_{k}\overline{q}F\big]\big\|_{p}\le C\|q\|_{H^{s,s}}\|F\|_{L^r}.
$$
These inequalities would in turn follow from
\begin{equation}\label{cuns}
\|qF\|_t\le C\|q\|_{H^{s,s}}\|F\|_{L^r}
\end{equation}
with $\frac{1}{2}=\frac{1}{t}-\frac{1}{p}$ and with $\frac{1-s}{2}=\frac{1}{t}-\frac{1}{p}$. To see that the second inequality follows from \eqref{cuns}, we note that the Fourier multiplier associated to $D^s\partial_{z}^{-1}$ can be written as
$$
\frac{|\xi|^{s}}{i\xi_1+\xi_2}=\frac{|\xi|^s(\xi_2-i\xi_1)}{|\xi|^2}=\frac{\xi_2-i\xi_1}{|\xi|}\frac{1}{|\xi|^{1-s}}
$$
so that the expected bounds for $D^s\partial_{z}^{-1}$ hold by the $L^{p}$-boundedness of the Riesz transforms followed by the Hardy--Littlewood--Sobolev inequality.

Now, by H\"older's inequality, \eqref{cuns} would  be a consequence of
\begin{equation}\label{abov}
\|q\|_{t(r/t)'}\le C\|q\|_{H^{s,s}}.
\end{equation}
Note that
$$
t(r/t)'=\frac{1}{\frac{1}{t}-\frac{1}{r}}\le\frac{2}{1-s}
$$
so that if $t(r/t)'\ge 2$ we can employ the Hardy--Littlewood--Sobolev inequality again to get the result. On the other hand, if $t(r/t)'< 2$ we will use H\"older's inequality to prove the estimate and the worst case is when $\frac{1}{2}=\frac{1}{t}-\frac{1}{p}$. Indeed, by writing 
$$
\|q\|^{t(r/t)'}_{t(r/t)'}=\int |q(z)|^{t(r/t)'}\frac{(1+|z|^2)^\alpha}{(1+|z|^2)^\alpha}\,dz
$$
and applying H\"older's inequality we obtain \eqref{abov} whenever
$$
\Big(\frac{2}{t(r/t)'}\Big)'\alpha>1, 
$$
where $\alpha=s\frac{t(r/t)'}{2}$. This is equivalent to the condition $s+1>2(\frac{1}{t}-\frac{1}{r})$, which is true as long as $s>2(\frac1p-\frac1r)$. 
\end{proof}

One can also obtain $L^\infty$-estimates that take the following form.

\begin{lem}\label{isitgood2} Let $p>2$. Then
$$
\|\S^k_q [F]\|_{L^\infty}\le C|k|^{-1}\big(\|q\|_{W^{1,p}}+\|q\|_{W^{1,p'}}\big)^2\|F\|_{L^{\infty}}.
$$
\end{lem}

\begin{proof} In the previous proof, we saw that
$$\big|\partial_{\overline{z}}^{-1}\big[e_{-k}G\big]\big|\le \frac{2}{|k|}\Big(\big|\partial_{\overline{z}}^{-1}\big[e_{-k}\partial_{\overline{z}} G\big]\big|+|G|\Big),
$$
so by \cite[Theorems 4.3.11]{AIM}, 
\begin{equation}\label{fr}\big\|\partial_{\overline{z}}^{-1}\big[e_{-k}G\big]\big\|_\infty\le \frac{2}{|k|}\Big(\|\partial_{\overline{z}} G\|_{p}+\|\partial_{\overline{z}} G\|_{p'}+\|G\|_\infty\Big).
\end{equation}
Taking $G=q\,\partial_{z}^{-1}\big[e_{k}\overline{q}F\big]$, we see that
$$\|\partial_{\overline{z}}G\|_p\le \|\partial_{\overline{z}}q\|_p\|\partial_{z}^{-1}\big[e_{k}\overline{q}F\big]\|_\infty+\|q\|_\infty\|\partial_{\overline{z}}\partial_{z}^{-1}\big[e_{k}\overline{q}F\big]\|_p,
$$
so by a further application of \cite[Theorems 4.3.11]{AIM}, the Hardy--Littlewood--Sobolev inequality, and the boundedness of the Beurling transform, we can deal with the first term on the right hand side of \eqref{fr}. The second and third terms are dealt with in a similar fashion.
\end{proof}

\begin{rem}\label{rem} 
The previous lemma yields the asymptotics in the $k$ variable of $(u_1,u_2)$  for Schwartz $q$. Namely,
recalling that  $u_1=(\I-\S^k_{q})^{-1}[1]\in L^\infty$, by Neumann series and the previous lemma, 
we obtain
$u_1(z,k)\to 1$ as $|k|\to \infty$.  Together with an application of the Riemann--Lebesgue lemma to $(z-\cdot)^{-1}q 
\in L^1$ this also yields that $u_{2}
=\partial_{\overline{z}}^{-1}[e_{-k}q\overline{u_{1}}]\to 0$  as $|k|\to\infty$.
\end{rem}

Armed with the Lemma~\ref{isitgood}, we can now follow the proof of Lemma~{4.1} in \cite{perry} to prove that $\mathcal{F}:H^{s,s}\to L^2$ is continuous.

\begin{thm}\label{one} Let $s\in(0,1)$ and $q_1,q_2\in H^{s,s}$. Then there is a function $\,C:\mathbb{R}_+^2\to \mathbb{R}_+$, bounded on compact subsets of its domain, such that
$$
\big\|\FF [q_2]-\FF [q_1]\big\|_2\le C(s,\|q_1\|_{H^{s,s}})C(s,\|q_2\|_{H^{s,s}})\big\|q_2-q_1\big\|_{H^{s,s}}.
$$
\end{thm}

\begin{proof}  As we have Lipschitz regularity for the linear Fourier transform, by the triangle inequality we can replace $u_1$ with $u_1-1$ in \eqref{kuusi}. Using the expansion
\begin{equation}\label{expansion}
(\I-\S^k_{q})^{-1}=\sum_{n=0}^{N-1} (\S^k_{q})^n+(\S^k_{q})^{N}(\I-\S^k_{q})^{-1},
\end{equation}
we write $u_1-1=\sum_{n=1}^{N} u_q^n$, where $u_q^n(\cdot,k)=(\S^k_{q})^n[1]$ for $1\le n\le N-1$, and the remainder term $u_q^{N}$ is defined by
$$u_q^{N}(\cdot,k)=(\S^k_{q})^{N}(\I-\S^k_{q})^{-1}[1].$$
By the triangle inequality, it will suffice to prove the lemma with $\FF [q_2]-\FF [q_1]$ replaced by $\FF_n [q_2]-\FF_n [q_1]$ defined by
$$
\FF_n[q](k)=\int_{\R^2} e_{-k}(z)\,q(z)\,  \overline{u_q^n(z,k)}\, dz. 
$$
For $n=1,\ldots,N-1$, this was proven by Perry \cite[Lemma 4.6]{perry} by explicitly writing out the iterated operators, and applying the multilinear estimate of Brown~\cite{brown} (see also the deep $L^p$--versions in \cite{B0,B1}). These terms can be bounded with $q\in L^2$.
Thus, we are left to deal with the remainder term.  

As long as $N>2/s^2$, we can use an application of the third inequality in Lemma~\ref{isitgood} followed by $N-1$ applications of the second inequality to prove 
\begin{equation}\label{this}
\|(\S^k_{q})^{N} [F]\|_{L^p}\le C(1+|k|)^{-Ns}\|q\|^{2N}_{H^{s,s}}\|F\|_\infty.
\end{equation}

Now, for local Lipschitz continuity, we write $\FF_N [q_2]-\FF_N [q_1]$ as
\begin{equation*}\label{thetwo}
\int_{\R^2} e_{-k}(z)\,q_2(z)\,  \big(\overline{u_{q_2}^N(z,k)}-\overline{u_{q_1}^N(z,k)}\big)\, dz\,+\int_{\R^2} e_{-k}(z)\,\big(q_2(z)-q_1(z)\big)\,  \overline{u_{q_1}^N(z,k)}.
\end{equation*}
By H\"older's inequality, the second integral is bounded by
\begin{align*}
\|q_2-q_1\|_{p'}\big\|(\S^k_{q_1})^{N}(\I-\S^k_{q_1})^{-1}[1]\big\|_p.
\end{align*}
Letting $2<p<\frac{2}{1-s}$, one can use H\"older's inequality again, first multiplying and dividing by a weight as in the proof of the previous lemma, in order to show that $\|q_2-q_1\|_{p'}\le C\|q_2-q_1\|_{H^{s,s}}$. Combining this with \eqref{this} we can bound the second integral by a constant multiple of
\begin{align*}
(1+|k|)^{-Ns}\|q_1\|^{2N}_{H^{s,s}}\|q_2-q_1\|_{H^{s,s}},
\end{align*}
which can be squared and integrated with respect to the $k$-variable to achieve the desired bound.

Similarly, to deal with the first integral, it will suffice to bound
$$
\big\|u_{q_2}^N(\cdot,k)-u_{q_1}^N(\cdot,k)\big\|_p=\Big\|(\S^k_{{q_2}})^{N}(\I-\S^k_{{q_2}})^{-1} [1]-(\S^k_{{q_1}})^{N}(\I-\S^k_{{q_1}})^{-1} [1]\Big\|_p.
$$
However, by the triangle inequality, this is bounded by the sum of
\begin{equation}\label{fd}
\Big\|\Big((\S^k_{{q_2}})^{N} -(\S^k_{{q_1}})^{N}\Big)(\I-\S^k_{{q_2}})^{-1}[1]\Big\|_p
\end{equation}
and
\begin{align}\label{sd}
&\Big\|(\S^k_{{q_1}})^{N}\Big[(\I-\S^k_{{q_2}})^{-1}-(\I-\S^k_{{q_1}})^{-1}\Big] [1]\Big\|_p\\\nonumber
=\,\,&\Big\|\Big(\S^k_{{q_2}}-\S^k_{{q_1}}\Big)(\S^k_{{q_1}})^{N}(\I-\S^k_{{q_2}})^{-1}(\I-\S^k_{{q_1}})^{-1} [1]\Big\|_p,
\end{align}
where in the final equality we have used the second resolvent identity. Now, for a fixed $F$ the terms $(\S^k_{{q}})^{N}  [F]$ and $\S^k_{q}[F]$ can be written as multilinear operators in $q$. Then,  writing $B_1[q,q,F]=\S^k_{{q}}[F]$, from the proof of the first estimate in Lemma~\ref{isitgood} we also have that
\begin{equation}\label{first}
\|B_1[q_1,q_2,F]\|_{L^p}\le C(1+|k|)^{-s}\|q_1\|_{H^{s,s}}\|q_2\|_{H^{s,s}}\|F\|_{L^{p}}
\end{equation}
and similarly, by writing $B_N[q,\ldots,q,F]=(\S^k_{{q}})^N[F]$ and adapting the proof of \eqref{this} we obtain
\begin{equation}\label{second}
\|B_N[q_1,\ldots,q_{2N},F]\|_{L^p}\le C(1+|k|)^{-Ns}\|q_1\|_{H^{s,s}}\ldots\|q_{2N}\|_{H^{s,s}}\|F\|_\infty.
\end{equation}
Writing the difference terms in the form
\begin{multline*}
B_2[q_2,q_2,q_2,q_2,F]-B_2[q_1,q_1,q_1,q_1,F]\\
=\,B_2[q_2-q_1,q_2,q_2,q_2,F]+B_2[q_1,q_2-q_1,q_2,q_2,F]
\, \\+B_2[q_1,q_1,q_2-q_1,q_2,F]+B_2[q_1,q_1,q_1,q_2-q_1,F],
\end{multline*}
the estimate for $\|\FF_N [q_2]-\FF_N [q_1]\|_p$ then follows by bounding \eqref{fd} and \eqref{sd} by combining \eqref{this}, \eqref{first}, \eqref{second} and the fact that $(\I-\S^k_{{q}})$ is bounded in $L^\infty$.
\end{proof}

\noindent{\it Added in Proof.} This final result was recently improved in \cite{BOP}, where they show that $\mathcal{F}$ is in fact locally Lipschitz continuous from $H^{s_1,s_2}$ to $H^{s_2,s_1}$ with $0<s_1,s_2<1$.

\end{document}